\documentclass[12pt.]{amsart}
\usepackage{amssymb}
\usepackage{lmodern}
\usepackage[T1]{fontenc}
\usepackage[headings]{fullpage}
\usepackage{paralist}
\usepackage[usenames]{color}
\usepackage{xspace}
\usepackage[normalem]{ulem}
\usepackage[colorlinks,pagebackref=true,pdftex]{hyperref}
\usepackage[alphabetic,abbrev,backrefs,lite]{amsrefs}

\usepackage{graphicx}

\newtheorem{Theorem}{Theorem}[section]
\newtheorem{Lemma}[Theorem]{Lemma}

\theoremstyle{definition}
\newtheorem{Remark}[Theorem]{Remark}

\newtheorem{Example}[Theorem]{Example}

\newtheorem{Definition}[Theorem]{Definition}

\theoremstyle{plain}

\begin{document}

\Large

\title{On the Koszulness of multi-Rees algebras of certain strongly stable ideals}
\author[Gabriel Sosa]{Gabriel Sosa}
\address{Department of Mathematics, Purdue University 150 N. University Street, West Lafayette, IN 47907-2067}
\email{gsosa@math.purdue.edu}

\begin{abstract}
Generalizing techniques that prove that Veronese subrings are Koszul, we show that Rees and multi-Rees algebras of certain types of principal strongly stable ideals are Koszul. 
We provide explicit Gr\"obner basis for the defining ideals of these multi-Rees algebras with squarefree initial monomials, to show that they are also normal Cohen-Macaulay domains. 
\end{abstract}

\maketitle

\section*{Introduction}

Let $K$ be a field and $R$ a standard graded $K$-algebra, i.e. $R= \bigoplus_{i\geq 0} R_{i}$, (where each $R_i$ is a $K$-vector space), $R_0=K$, $\text{dim}_K(R_1)=s \in \mathbb{N}$ and $R_iR_j=R_{i+j}$, for all $i,j$. By letting $\{b_1, \dots, b_s\}$ be a $K$-basis for $R_1$, we can consider a graded surjective map $\varphi: K[T_1,T_2,\dots,T_s] \rightarrow R$, given by the rule $\varphi(T_i)=b_i$, and obtain that $R \cong K[T_1,\dots,T_s]/\text{Ker}\varphi$. We call the graded ideal $I=\text{Ker}\varphi$, the defining ideal of $R$. We say that $R$ is Koszul if its residue field $K = R/\mathfrak{m}$, where $\mathfrak{m}$ is the graded maximal ideal of $R$, has a linear free $R$-resolution (i.e. the nonzero entries of the matrices describing the maps in the free $R$-resolution of $K$ are linear forms). Koszul Algebras were introduced by Priddy, in \cite{P}. Fr\"oberg, proved in \cite{F} that if the defining ideal of $R$ is generated by quadratic monomials, then $R$ is Koszul. Fr\"oberg's result coupled with a deformation argument provides a sufficient, yet not necessary, condition for the Koszulness of $R$: the existence of a monomial order on $K[T_1,\dots,T_s]$ such that $I$ has a quadratic Gr\"obner basis, in this case, it is said that $I$ satisfies the $G$-quadratic condition (the $G$-quadratic condition is not necessary for the Koszulness of $R$ as evidenced by examples produced by Hibi and Ohsugi \cite{HO} and Caviglia \cite{Ca}).

Showing that the defining ideal satisfies the $G$-quadratic condition remains a useful technique to determine the Koszulness of standard graded algebras. In their survey paper, Conca, De Negri and Rossi \cite{CDR} present a proof that the $d^{\text{th}}$ Veronese subring, the subring of $K[X_1, \dots, X_n]$ generated by all monomials of degree $d$, is Koszul (the Gr\"obner basis for the defining ideal is defined using a concept that we will refer to as ordering). In his book, Sturmfels \cite{S}, proves that Veronese type algebras, $K[\mathcal{N}]$, with $\mathcal{N}=\{X_1^{\beta_1}X_2^{\beta_2} \cdots X_n^{\beta_n}: \beta_1+\beta_2+\dots+\beta_n=d \text{ and } \beta_i \leq \alpha_i, \text{ for } 1 \leq i \leq n \}$  and $d, \alpha_1, \dots, \alpha_n $ fixed nonnegative integers, are Koszul, by finding squarefree quadratic Gr\"obner bases for their defining ideals and introducing the concept of sorting. A squarefree Gr\"obner basis for the defining ideal of $K[\mathcal{N}]$, when $\mathcal{N}$ is an affine semigroup, guarantees that $K[\mathcal{N}]$ is also a Cohen-Macaulay normal domain because of results by Sturmfels \cite{S} and Hochster \cite{Hoc}.  De Negri \cite{D} proved that the defining ideals for subalgebras of Veronese type generated by principal strongly stable sets, and lex-segments, also possess a squarefree Gr\"obner basis according to a different monomial ordering.

Strongly stable ideals play a pivotal role on computational commutative algebra because they are in correspondance with Borel-fixed ideals when $\text{char}(K)=0$, and this fact links their study to that of generic initial ideals and their extremal properties as seen in Galligo \cite{Ga}, Bayer and Stillman \cite{BS1}, \cite{BS2}, Eliahou and Kervaire \cite{EK} and Conca \cite{Co}. In their paper, Francisco, Mermin and Schweig \cite{FrMSc} study many properties of strongly stable ideals, amongst them a combinatorial interpretation of the Betti numbers of principal strongly stable ideals in $\text{char}(K)=0$  . 

Recently, examples of Rees, and multi-Rees, algebras that are Koszul, Cohen-Macaulay normal domains, because of the $G$-quadratic condition, have surfaced: Rees Algebras of Hibi ideals, by Ene, Herzog and Mohammadi \cite{EHM} and the multi-Rees algebra of the direct sum of powers of the maximal ideal by Lin and Polini \cite{LPo}. 

In Section 1, we introduce the concept of \textit{monomial sets closed under comparability} and use it to establish a monomial order that will allow ideals, generated by differences between incomparable products of monomials and their corresponding comparable product, to have squarefree quadratic Gr\"obner bases.

In Section 2, we prove that these squarefree quadratic Gr\"obner bases are the defining ideals for certain types of Rees and multi-Rees algebras, characterize sets closed under comparability in terms of principal strongly stable ideals and present examples of algebras that are Koszul Cohen-Macaulay normal domains that had not been previously considered.

I would like to thank my advisor, Giulio Caviglia, for his guidance on the material presented and his suggestions regarding the structure of this paper.

\section{Preliminaries}

Let \(S=K[X_{1},X_{2},\dots,X_{n}]\) be a polynomial ring over a field \(K\). We consider the standard grading of $S$ as a $K$-algebra, (i.e. $S= \bigoplus_{i\geq 0} S_{i}$, with a $K$-basis for $S_{i}$ given by all monomials in $S$ of degree $i$, denoted $\text{Mon}(S_i) $). Additionally, we establish the following order for  the elements of $\mathbb{N}^2$: given $\mathbf{a}, \mathbf{b} \in \mathbb{N}^2$, we say that $\mathbf{a} <_{lex} \mathbf{ b}$ if the first nonzero entry of $\mathbf{b-a}$ is positive, and the following order for the variables of $S$, ($X_1 >  X_2 > \cdots > X_n$).

For a monomial $u$ with $\text{deg}(u)=d$, we will use the \textbf{standard factorization} of $u=X_{1}^{(u)}X_{2}^{(u)}\cdots X_{d}^{(u)}$, where $X_{s}^{(u)} \in \{X_1, \dots, X_n \}$, for all $1 \leq s \leq d$, and $X_{1}^{(u)} \geq X_{2}^{(u)} \geq \dots \geq X_{d}^{(u)}$. We will define $\text{max}(u)=X_{1}^{(u)}$, and $\text{min}(u)=X_{d}^{(u)}$, which makes the fact that $\text{max}(u) \geq X_{s}^{(u)} \geq \text{min}(u)$, for all $1 \leq s \leq d$, natural.

\begin{Definition}\label{1.1.} Let $p \leq q$, $u \in \text{Mon}(S_p)$ and $v \in \text{Mon}(S_q)$.  We define:

\begin{itemize}[-]
\item The \textbf{ordering} of $u,v$, denoted \textbf{ord}($u,v$), as the pair $(\hat{u},\hat{v})=\left( \prod_{j=q+1}^{q+p}X_{j}^{(uv)}, \prod_{j=1}^{q }X_{j}^{(uv)}\right)$ 

\item If $p=q$, the \textbf{sorting} of $u,v$, denoted \textbf{sort}($u,v$) as $(\overline{u},\overline{v})=\left(\prod_{j=1}^p X_{2j-1}^{(uv)}, \prod_{j=1}^{p} X_{2j}^{(uv)}\right)$
\end{itemize}

\end{Definition}

Notice that $\text{deg}(u)=\text{deg}(\hat{u})=\text{deg}(\overline{u})$, $\text{deg}(v)=\text{deg}(\hat{v})=\text{deg}(\overline{v})$ and $uv=\hat{u}\hat{v}=\overline{u}\overline{v}$. Additionally $\text{max}(\hat{u}) \leq \text{min}(\hat{v})$ and $\overline{u} >_{rev} \overline{v}$. Recall that  $u >_{rev} v$ if either (a) $\text{deg}(u) > \text{deg}(v)$, or, (b) if  $\text{deg}(u)=\text{deg}(v)=d$ and there is $1 \leq k \leq d$ such that  $X_{k}^{(u)}>X_{k}^{(v)}$ and $X_{k'}^{(u)} = X_{k'}^{(v)}$ for all $k'>k$.

As an example, for $u=X_1X_1X_3 $ and $v=X_2X_3X_4$, $\textbf{ord}(u,v)= (X_3X_3X_4,X_1X_1X_2)$ and $\textbf{sort}(u,v)= (X_1X_2X_3,X_1X_3X_4)$.

\

The new examples of Koszul, Cohen-Macaulay normal domains introduced in this paper mainly deal with Rees and multi-Rees algebras of ideals of $S$. For this reason we will recall their structure: Given ideals $I_1, I_2, ..., I_s \subset S$, consider $M=I_1 \oplus I_2 \oplus \dots \oplus I_s$, the multi-Rees algebra of $I_1, \dots, I_s$ over $S$, denoted $\mathcal{R}(M)$, is the multigraded algebra:

\[ \mathcal{R}(M)=\bigoplus_{(a_1,\dots,a_s) \in \mathbb{N}^s} I_1^{a_1}t_1^{a_1} \cdots I_s^{a_s}t_r^{a_s} \subset S[t_1,\dots, t_s] \cong K[X_1,\dots,X_n, t_1, \dots, t_s].\]

We will restrict ourselves to the case where each $I_i$ is a homogeneous monomial ideal with minimal generators of a fixed degree $d_i$, with the requirement $d_1 \leq d_2 \leq \dots \leq d_s$.  We will index the minimal homogeneous generators of $I_i$ in the following way $u_{i1} >_{rev} u_{i2} >_{rev} \dots >_{rev} u_{in_i}$. And to ease notation in the latter parts of this paper we will let $n_0=n$,  and declare $u_{0j}=X_j$ for all $1 \leq j \leq n_0$. 

Consider the following sets $\mathcal{M}=\coprod_{i=0}^{s} \{u_{ij}:1 \leq j \leq n_i \} $, $\mathcal{N}' = \{u_{ij}t_{i}: 1 \leq i \leq s \text{ and } 1 \leq j \leq n_i\}$, and  $\mathcal{N}=\mathcal{N}' \cup \{u_{0j}: 1 \leq j \leq n \}$ . With this notation,  we are able to define the following surjective map:
\[ \Psi: K[T_{ij}: u_{ij} \in \mathcal{M}] \longrightarrow K[\mathcal{N}] \cong S[\mathcal{N}'] = \mathcal{R}(M) \subset S[t_1,\dots, t_s] \cong K[X_1,\dots,X_n, t_1, \dots, t_s].\]

Given by $\Psi(T_{ij})=u_{ij}t_{i}$ for all $u_{ij} \in \mathcal{M}$, with $0 \leq i \leq s$ and $1 \leq j \leq n_{i}$ and letting $t_0=1$. This implies that $\mathcal{R}(M) \cong K[T_{ij}: u_{ij} \in \mathcal{M}]/\text{Ker}\Psi$. 

Our first goal is to define properties on $\mathcal{M}$ so that $\text{Ker}\Psi$ has a squarefree quadratic Gr\"obner basis. 

\begin{Definition}\label{1.2.} Let $u_{ij}, u_{i'j'} \in \mathcal{M}$, with $(i,j) <_{lex} (i',j')$. We say that $u_{ij}, u_{i'j'}$ are comparable, and denote it $u_{ij} \prec u_{i'j'}$, if:

(a) $i=i'$ and $(\overline{u_{ij}},\overline{u_{i'j'}})=(u_{ij}, u_{i'j'})$, or (b) $i<i'$ and $(\hat{u_{ij}},\hat{u_{i'j'}})=(u_{ij}, u_{i'j'})$.

By contrast, we will say that $u, v \in \mathcal{M}$ are \textbf{incomparable} if $u \nprec v$ and $v \nprec u$.

We will say that $\mathcal{M}$ is \textbf{closed under comparability} if for all $u_{ij}, u_{i'j'} \in \mathcal{M}$, with $(i,j) <_{lex} (i',j')$, there are elements $u_{il}, u_{i'l'}  \in \mathcal{M} \text{ such that } u_{ij} u_{i'j'}=u_{il}u_{i'l'}  \text{ and } u_{il} \prec u_{i'l'}$. 

An equivalent statement would be: 

\textit{ $\mathcal{M}$ is closed under comparability if for any $u_{ij}, u_{i'j'} \in \mathcal{M}$, with $(i,j) <_{lex} (i',j')$, we have that $\overline{u_{ij}},\overline{u_{i'j'}}  \in \mathcal{M}$, if $i=i'$, and $\hat{u_{ij}},\hat{u_{i'j'}} \in \mathcal{M}$, if $i<i'$}.

Additionally, to ease notation later, from this point on we will denote  $X_{k}^{(u_{ij})}$ by $Y_{ijk}$.
\end{Definition}

An example of a set that is closed under comparability is presented now:

\begin{Example}\label{1.3.} Let $n=4$. If $u \in \text{Mon}(S_d)$, then the \textbf{principal strongly stable set of $u$} is the set $\mathcal{B}(u)=\left\{v \in \text{Mon}(S): v = \frac{X_{i_1}}{X_{j_1}}\dots \frac{X_{i_t}}{X_{j_t}}u , i_k< j_k \text{ with } 1 \leq k \leq t \right\}$. A monomial ideal $I$ is principal strongly stable if its minimal monomial generating set is the principal strongly stable set for some $u \in I$. Consider the following principal strongly stable ideals $I_1=(\mathcal{B}(X_3X_4) )$, $I_2=(\mathcal{B}(X_2X_2X_3) )$, $ I_3=(\mathcal{B}(X_1X_2X_2))$, $I_4=(X_1^5)$.

We can determine the surjective map onto the multi-Rees algebra of $I_1, I_2, I_3, I_4$ over $S$, by defining the elements of $\mathcal{M}$ in the following manner: $d_0=1, d_1=2,  d_2=3, d_3=3, d_4=4$, 

$\begin{array}{cccccc}
u_{01}=X_1  & u_{11}=X_1X_1  & u_{16}=X_3X_3 & u_{21}=X_1X_1X_1 & u_{26}=X_1X_2X_3 & u_{31}=X_1X_1X_1 \\
u_{02}=X_2  & u_{12}=X_1X_2  & u_{17}=X_1X_4  & u_{22}=X_1X_1X_2 & u_{27}=X_2X_2X_3 & u_{32}=X_1X_1X_2\\
u_{03}=X_3  & u_{13}=X_2X_2 &  u_{18}=X_2X_4 & u_{23}=X_1X_2X_2  &  & u_{33}=X_1X_2X_2  \\
u_{04}=X_4  & u_{14}=X_1X_3 & u_{19}=X_3X_4 & u_{24}=X_2X_2X_2 &  & \\
&   u_{15}=X_2X_3 &  & u_{25}=X_1X_1X_3 &  & u_{41}=X_1X_1X_1X_1 \\

\end{array}$
\end{Example}

The following Lemma is the central result of this section, the remainder of the section will be devoted to providing definitions and remarks that will allow for its proof.

\begin{Lemma}\label{1.4.} Let $\mathcal{M}$ be closed under comparability. Consider the polynomial ring $S'=K[T_{ij}:u_{ij} \in \mathcal{M}]$. There is a monomial order $<_{\tau}$, such that the set 
\[\displaystyle G=\left\{ \underline{T_{ij}T_{i'j'}}-T_{il}T_{i'l'} \vert \ (i,j) <_{lex} (i',j'),  u_{ij}u_{i'j'}=u_{il}u_{i'l'}, u_{ij} \nprec u_{i'j'}, u_{il} \prec u_{i'l'} \right\}\]
is a Gr\"obner basis for $I=(G)$, and the underlined terms are initial terms according to $<_{\tau}$.
\end{Lemma}

It is clear that $(G) \subseteq \text{Ker} \Psi$, since $\Psi(T_{ij}T_{i'j'}-T_{il}T_{i'l'})=u_{ij}u_{i'j'}t_it_{i'}-u_{il}u_{i'l'}t_{i}t_{i'}$ and $u_{ij}u_{i'j'}=u_{il}u_{i'l'}$ for $T_{ij}T_{i'j'}-T_{il}T_{i'l'} \in G$. In the next section we will prove that $(G) =  \text{Ker} \Psi$.

We need to guarantee the existence of a monomial order, $<_{\tau}$, such that the underlined terms above are initial according to $<_{\tau}$ first. For this purpose we will fix our attention on the concepts in the following definition.

\begin{Definition} \label{1.5.} Let $\mathcal{F}=\{f_1, f_2, \dots, f_p \}$ be a set of polynomials in $S'$ such that there is a distinguished monomial $m_i \in \text{supp}(f_i)$ for all $1 \leq i \leq p$. To simplify notation we will assume that the coefficient of $m_i$ in $f_i$ is $1$. Let $f=\sum_{m \in \text{supp}(f)} b_mm $, we say that \textbf{$f \in S'$ reduces to $f'$, in one step, modulo $\mathcal{F}$}, denoted $f \rightarrow_{\mathcal{F}} f' $, if there are $1 \leq j \leq p$ and $m' \in \text{supp}(f)$ such that $m_j \mid m'$ and $\displaystyle f'=f -  \frac{b_{m'}m'}{m_j} (f_j)$. We say that $\rightarrow_{\mathcal{F}}$ is a \textbf{Noetherian reduction relation} if given any $f \in S$ only  a finite number of one step reductions modulo $\mathcal{F}$ can be carried out starting at $f$. \end{Definition} 
 
It is clear that if $\mathcal{G}$ is a Gr\"obner basis, then  $\rightarrow_{\mathcal{G}}$ is a Noetherian reduction relation. 
Sturmfels  wrote a partial converse for this statement using coherent markings.

Let $\mathcal{F}=\{f_1, f_2, \dots, f_p \}$ be a set of polynomials such that there is a distinguished monomial $m_i \in \text{supp}(f_i)$. We say that the set $\mathcal{F}$ is \textbf{marked coherently} if there is a monomial ordering $<_{\tau}$ such that $m_i= \text{in}_{<_{\tau}}(f_i)$.

\begin{Theorem} \label{1.6.} A finite set $\mathcal{F}$ is marked coherently if and only if $\rightarrow_{\mathcal{F}}$ is a Noetherian reduction relation.
\end{Theorem}
The proof of this statement appears in  \cite[Chapter 3]{S}.

To prove that  $\rightarrow_{\mathcal{G}}$, with $\mathcal{G}$ as in Lemma \ref{1.4.}, is a \textbf{Noetherian reduction relation}, we need to create a way to encode information for monomials in $S'$. Let $m \in \text{Mon}(S')$, then $m$ can be expressed as:
\[ m=\prod_{i=0}^{s}  \left(  \prod_{j=1}^{r_i} T_{i\sigma_{i}^{m}(j)} \right) \text{ with } u_{i\sigma_{i}^{m}(j)} \in \mathcal{M}, \text{ and } 1 \leq \sigma_{i}^{m}(j) \leq n_i \]
where, as usual, $\displaystyle \prod_{j=1}^{r_i} T_{i\sigma_{i}^{m}(j)}=1$ if $r_i=0$. 

We say that $ \prod_{j=1}^{r_i} T_{i\sigma_{i}^{m}(j)} $ is the \textbf{$i^\text{th}$ level} of $m$, and it can be represented as a $r_i \times d_i$-matrix, $A_i$, with entries given by  $\left( A_i \right)_{jk}=Y_{i\sigma_{i}^{m}(j)k}$,  with the understanding that if $r_i =0$ then the matrix is empty. If we consider the set $\mathcal{M}$ in Example \ref{1.3.}, and the monomial $m=T_{32}T_{32}T_{33}$, we obtain that $A_0,A_1,A_2,A_4$ are empty and that $A_3=\begin{bmatrix} X_1 & X_1 & X_2 \\ X_1 & X_1 & X_2 \\ X_1 & X_2 & X_2 \end{bmatrix}$; but if we had expressed $m$ as $T_{32}T_{33}T_{32}$, then $A_3=\begin{bmatrix} X_1 & X_1 & X_2 \\ X_1 & X_2 & X_2 \\ X_1 & X_1 & X_2  \end{bmatrix}$.

This indicates that  the representation of the $i^\text{th}$-level of $m$ as a matrix is not unique, (it depends on the way that the product is expressed). To explain our choice of representation, we need to define the \textbf{number of inversions of a $c  \times d$-matrix $A$}, denoted $ e_{A}$, as
\[e_{A}=\sum_{i=1}^{c} \left( \sum_{j=1}^{d} \left\vert \left\{ \left( i',j' \right) : (j,i) <_{lex} (j',i' )\text{ and } A_{i'j'} > A_{ij} \right\} \right\vert \right).\]
In other words, for every entry $A_{ij}$, let $e_{A,ij}$ be the number of entries in $A$ that are larger than $A_{ij}$, and are located either to the right of $A_{ij}$, or, in the same column ($j$-th column) as $A_{ij}$, but below it; $e_{A}$ is the addition of the  $e_{A,ij}$ over all the entries in $A$. 

We say that a $c  \times d$-matrix $B$ is \textbf{inversion minimal}, if $B_{jk} \leq B_{jk'}$, whenever  $1 \leq k < k' \leq d$, (for all $1 \leq j \leq c$), and  $e_B \leq e_{B'}$ for all $B'$ obtained by permuting the rows of $B$. We will denote the set of all inversion minimal matrices with entries in $S$, and size $c \times d_i$, for some $c \in \mathbb{N}$, by $\mathcal{T}_{d_i}$.

With these stipulations in place, we realize a monomial $m$ in $S'$ as a $(s+1)$-tuple of inversion minimal matrices in $\mathcal{T}_{d_s} \times \mathcal{T}_{d_{s-1}} \times \dots \times \mathcal{T}_{d_{1}} \times \mathcal{T}_{d_0}$, in the following manner:
\[m \longrightarrow (B_s^{(m)},B_{s-1}^{(m)},\dots,B_1^{(m)},B_0^{(m)}),\text{ where } B_{i}^{(m)} \text{ is an inversion minimal } r_i \times d_i \text{ -matrix}, \] representing the $i^\text{th}$  level of $m$.

We can now define the \textbf{number of inversions at the $i^{\text{th}}$ level of $m$}, denoted $ e_{m,i}$, as $e_{m,i}=e_{B_i^{(m)}}$, (i.e. the number of inversions of an inversion minimal matrix $B_i^{(m)}$ for the $i^{\text{th}}$ level of $m$).

The \textbf{total number of inversions of $m$} will be denoted $e_m := \sum_{i=0}^{s} e_{m,i}$.

We define the \textbf{comparability number of $m$}, denoted $c_m$, as 
\[c_m= \sum_{i=0}^{s} \left( \sum_{j=1}^{r_i} \left(\sum_{k=1}^{d_i} \left\vert \left\{ \left( i',\sigma_{i'}^{m}(j'),k' \right) :  i'>i \text{ and } Y_{i'\sigma_{i'}^{m}(j')k'} < Y_{i\sigma_{i}^{m}(j)k} \right\} \right\vert \right)\right).\]
We can realize what $c_m$ means in terms of the entries in any $(s+1)$-tuple of matrices $(A_s, \dots, A_1, A_0)$, with the property that $A_i$ is a matrix representation for the $i^{\text{th}}$-level of $m$, (in particular, for a representation of $m$ in terms of inversion minimal matrices $(B_s^{(m)},\dots, B_0^{(m)})$): For every entry $(A_{i})_{jk}$, let $c_{m,A_i,jk}$ be the number of entries in matrices $A_{i'}$, with $i'>i$, that are smaller than $(A_i)_{jk}$; $c_m$ is the addition of the $c_{m,A_i,jk}$ over all the entries in the $(s+1)$-tuple $(A_r,\dots,A_1, A_0)$.

The pair $(c_m,e_m)$ will be defined as the \textbf{level of reduction} of $m$. If $c_m=e_m=0$ we say that $m$ is \textbf{completely reduced}. 

If we consider the set $\mathcal{M}$ in Example \ref{1.3.}, then  $\begin{bmatrix} \text{ } \end{bmatrix} \begin{bmatrix} \text{ } \end{bmatrix} \begin{bmatrix}  X_1 & X_1 & X_2 \\ \underline{X_2} & X_2 & X_3 \end{bmatrix}  \begin{bmatrix} \mathbf{X_1} & \mathbf{X_1} \\ \underline{\mathbf{X_2}} & \underline{X_4} \\ \underline{X_3} & X_3  \end{bmatrix}  \begin{bmatrix} \mathbf{X_1} \\ \mathbf{X_1} \\X_4 \end{bmatrix}$ is a representation for the monomial $m=T_{01}^2T_{04}T_{11}T_{16}T_{18}T_{22}T_{27}$ in terms of inversion minimal matrices. The bolded variables make contributions to the \textbf{comparability number} of $m$, while the underlined variables make contributions to the \textbf{total number of inversions} of $m$. The \textbf{level of reduction} in this case is $(25,  4)$. 

\begin{Remark} \label{1.7.} Let $\mathcal{M}$ be a set closed under comparability.
\begin{enumerate}[(i)]

\item  If $u_{ij} \nprec u_{i'j'}$ with $i<i'$, and, $T_{ij}T_{i'j'}$ divides $m$; then $c_m>c_{m'}$ where $m'=\frac{T_{il}T_{i'l'}}{T_{ij}T_{i'j'}}m$, and $(u_{il},u_{i'l'})=\text{ord}(u_{ij},u_{i'j'})$. This implies that $(c_{m'},e_{m'}) <_{lex} (c_m,e_m)$.

\item  If $A$ is a matrix such that there are entries $A_{j_1k_1} < A_{j_2k_2}$, with $(k_1,j_1) <_{lex} (k_2,j_2)$, and we consider the matrix $A'$ such that $A'_{jk}=A_{jk}$ whenever $(j,k) \neq (j_1,k_1), (j_2,k_2)$, $A'_{j_1k_1}=A_{j_2k_2}$ and $A'_{j_2k_2}=A_{j_1k_1}$, then $e_{A'}<e_{A}$.

\item  If $u_{ij} \nprec u_{ij'}$ and $u_{ij'} \nprec u_{ij}$, and, $T_{ij}T_{ij'}$ divides $m$; then $c_m=c_{m'}$, where $m'=\frac{T_{il}T_{il'}}{T_{ij}T_{ij'}}m$, and $(u_{il},u_{il'})=\text{sort}(u_{ij},u_{ij'})$. Additionally $e_{m',i}<e_{m,i}$ and $e_{m',i'}=e_{m,i'}$, for all $i' \neq i$, which implies $e_{m'} < e_{m}$, and then logically $(c_{m'},e_{m'}) <_{lex} (c_m,e_m)$.

\item A monomial $m \in S'$ is completely reduced if and only if $u_{ij} \prec u_{i'j'}$ or $u_{i'j'} \prec u_{ij}$, whenever $T_{ij}T_{i'j'}$ divides $m$.

\end{enumerate}
\end{Remark}

\begin{proof}

\noindent (i)  Since $u_{ij} \nprec u_{i'j'}$, it must happen that $Y_{ij1}=\text{max}(u_{ij}) > \text{min}(u_{i'j'})=Y_{i'j'd_{i'}}$.

Let $p=\text{max}\{ k: 1 \leq k \leq d_i \text{ and } Y_{ijk}>Y_{i'j'(d_{i'}-k+1)}\}$. Consider:
\[u= \left(  \prod_{k=p+1}^{d_i} Y_{ijk} \right)\left(  \prod_{k=1}^{p} Y_{i'j'(d_{i'}-k+1)} \right) \text{ and } v=\left(  \prod_{k=1}^{d_{i'}-p} Y_{i'j'k} \right)\left(  \prod_{k=1}^{p} Y_{ijk} \right) \]

Realize that $\text{max}(u)=\text{max}(Y_{ij(p+1)},Y_{i'j'(d_{i'}-p+1)})$, while $\text{min}(v)=\text{min}(Y_{ijp},Y_{i'j'(d_{i'}-p)})$. The definition of $p$ implies that $Y_{i'j'(d_{i'}-p)} \geq Y_{ij(p+1)}$ and $ Y_{ijp}>Y_{i'j'(d_{i'}-p+1)}$. The former statement coupled with the fact that $Y_{ijp} \geq  Y_{ij(p+1)}$ and $Y_{i'j'(d_{i'}-p)} \geq  Y_{i'j'(d_{i'}-p+1)}$, proves that $\text{min}(v) \geq \text{max}(u)$. Hence $u = u_{il}$ and $v=u_{i'l'}$.

Since $c_m$ and $c_{m'}$ are independent of whether their $i^{\text{th}}$-levels are represented by inversion minimal matrices or not, we will consider representations $(A_r, \dots, A_1, A_0)$ and $(A'_r,\dots,A'_1,A'_0)$, for $m$ and $m'$ respectively, such that  $A_{\hat{i}}=A'_{\hat{i}}$ for $\hat{i} \neq i, i'$, $(A_{i})_{\hat{j}\hat{k}}=(A'_{i})_{\hat{j}\hat{k}} \text{ and } (A_{i'})_{\hat{j}\hat{k}}=(A'_{i'})_{\hat{j}\hat{k}}$ whenever $\hat{j}>1$, and, $(A_i)_{1k}=Y_{ijk}, (A_{i'})_{1k}=Y_{i'j'k}, (A'_{i})_{1k}=Y_{ilk} \text{ and } (A'_{i'})_{1k}=Y_{i'l'k}$. 

Notice that $(A_{i'})_{1k} = Y_{i'j'k }\leq Y_{i'l'k} = (A'_{i'})_{1k}$ for all $1 \leq k \leq d_{i'}$, since $(Y_{i'l'1}, Y_{i'l'2},\dots, Y_{i'l'd_{i'}})$ is a rearrangement of $(Y_{i'j'1}, \dots, Y_{i'j'(d_{i'}-p)},Y_{ij1}, \dots, Y_{ijp})$,  with entries written in decreasing order, and $Y_{ijk} \geq Y_{ijp} > Y_{i'j'(d_{i'}-p+1)} \geq Y_{i'j'(d_{i'}-p+k)}$ for $1 \leq k \leq p$.

Furthermore, for $p+1 \leq k \leq d_{i}$, there exists a $\pi(k)$ such that $(A_{i})_{1k}=(A'_{i})_{1\pi(k)}$; for $1 \leq k \leq d_{i'}-p$, there exists a $\pi'(k)$ such that $(A_{i'})_{1k} =(A'_{i'})_{1\pi'(k)}$; and, for $1 \leq k \leq p$, there exist $\gamma(k)$ and $\gamma'(k)$ such that $(A_{i})_{1k}=(A'_{i'})_{1\gamma'(k)}$ and $(A_{i'})_{1(d_{i'}-k+1)}=(A'_{i})_{1\gamma(k)}$, and $\{ \pi(k) \} \cap \{ \gamma(k) \} =\{ \pi'(k) \} \cap \{ \gamma'(k) \} = \o$ . 

We will proceed to compare the contributions that the entries in $(A_s, \dots, A_1, A_0)$ make to the comparability number of $m$ against the contributions that the entries in the arrangement $(A'_s,\dots,A'_1,A'_0)$ make to the comparability number of $m'$. 

\noindent (a) 
$c_{m',A'_{\hat{i}},\hat{j}\hat{k}} = c_{m,A_{\hat{i}},\hat{j}\hat{k}}$ whenever $\hat{i}< i \text{ or } \hat{i} >i'$. Since $(A_{\hat{i}})_{\hat{j}\hat{k}}=(A'_{\hat{i}})_{\hat{j}\hat{k}}$ , and the entries in the arrangement $(A'_s, \dots, A'_{\hat{i}+2}, A'_{\hat{i}+1})$ are the same entries as in the arrangement $(A_s, \dots,  A_{\hat{i}+2}, A_{\hat{i}+1})$ after a permutation.

\noindent (b)
$c_{m',A'_{\hat{i}},\hat{j}\hat{k}} \leq  c_{m,A_{\hat{i}},\hat{j}\hat{k}}$ whenever $i< \hat{i} < i'$, or, $\hat{i} = i,i'$ and $\hat{j}>1$. Since $A_{\hat{i}}=A'_{\hat{i}}$, for $\hat{i} \neq i, i'$; $(A_{\hat{i}})_{\hat{j}\hat{k}}=(A'_{\hat{i}})_{\hat{j}\hat{k}}$, for $\hat{i}=i,i'$ and $\hat{j}>1$, and $(A_{i'})_{1k} \leq (A'_{i'})_{1k}$ for all $1 \leq k \leq d_{i'}$.

\noindent (c) $c_{m',A'_{i'},1\pi'(k)} = c_{m,A_{i'},1k} $, whenever $1 \leq k \leq d_{i'}-p$, by the same argument as in (a).

\noindent (d) $c_{m',A'_{i},1\pi(k)} \leq c_{m,A_{i},1k} $, whenever $p+1 \leq k \leq d_{i}$, by the same argument as in (b).

\noindent (e) $ c_{m',A'_i,1\gamma'(k)}+c_{m',A'_{i'},1\gamma(k)}<c_{m,A_i,1k}+c_{m,A_{i'},1(d_{i'}-k+1)}$, whenever $1 \leq k \leq p$. We will break down the explanation:

\begin{itemize}[-]
\item  If $(A'_{\hat{i}})_{\hat{j}\hat{k}} < (A'_{i'})_{1\gamma'(k)}$ with $ \hat{i} > i'$, then $(A_{\hat{i}})_{\hat{j}\hat{k}} < (A_{i})_{1k}$, since $A'_{\hat{i}}=A_{\hat{i}}$ (whenever $\hat{i} \neq i',i$), $(A'_{i'})_{1\gamma'(k)}=(A_{i})_{1k}$ and if $\hat{i}>i'$ then $\hat{i}>i$. Analogously, if $(A'_{\hat{i}})_{\hat{j}\hat{k}} < (A'_{i})_{1\gamma(k)}$ with $ \hat{i} > i'$, then $(A_{\hat{i}})_{\hat{j}\hat{k}} < (A_{i'})_{1(d_{i'}-k+1)}$.

\noindent This proves that if $(A'_{\hat{i}})_{\hat{j}\hat{k}}$ contributes to the growth of $c_{m',A'_i,1\gamma'(k)}+c_{m',A'_{i'},1\gamma(k)}$, when $\hat{i}>i'$, then $(A_{\hat{i}})_{\hat{j}\hat{k}}$ contributes to the growth of $c_{m,A_i,1k}+c_{m,A_{i'},1(d_{i'}-k+1)}$.

\item If $(A'_{\hat{i}})_{\hat{j}\hat{k}} < (A'_{i})_{1\gamma(k)} $ with $i' > \hat{i} > i$, or $\hat{i}=i'$ and $\hat{j}>1$, then $(A_{\hat{i}})_{\hat{j}\hat{k}} <(A_{i})_{1k}$. Since $(A'_{\hat{i}})_{\hat{j}\hat{k}}=(A_{\hat{i}})_{\hat{j}\hat{k}}$ (whenever $\hat{i} \neq i',i$, or, $\hat{i}=i'$ and $\hat{j}>1$), and $ (A'_{i})_{1\gamma(k)} = (A_{i'})_{1(d_{i'}-k+1)} = Y_{i'j'(d_{i'}-k+1)} < Y_{ijk}= (A_{i})_{1k}$.

\noindent This proves that if $(A'_{\hat{i}})_{\hat{j}\hat{k}}$ contributes to the growth of $c_{m',A'_i,1\gamma'(k)}+c_{m',A'_{i'},1\gamma(k)}$, when $i' > \hat{i} > i$, or $\hat{i}=i'$ and $\hat{j}>1$, then $(A_{\hat{i}})_{\hat{j}\hat{k}}$ contributes to the growth of $c_{m,A_i,1k}+c_{m,A_{i'},1(d_{i'}-k+1)}$.

\item Finally, we have that $(A'_{i'})_{1\hat{k}} = Y_{i'l'\hat{k}} \geq \text{min}(u_{i'l'}) \geq \text{max}(u_{il}) \geq Y_{ilk'}=(A'_{i})_{1k'}$,  for all $1 \leq k' \leq d_{i}$ and all $1 \leq \hat{k} \leq d_{i'} $; which implies that $(A'_{\hat{i}})_{1\hat{k}} $ makes no contribution to the growth $c_{m',A'_i,1\tau'(k)}+c_{m',A'_{i'},1\tau(k)}$, with $\hat{i}=i'$. 

\noindent On the other hand, for $d_{i'}-p+1 \leq \hat{k} \leq d_{i'} $, we have $(A_{i'})_{1\hat{k}} =Y_{i'j'\hat{k}} < Y_{ijk} = (A_{i})_{1k}$, for all $1 \leq k \leq d_i$, which means that $(A_{\hat{i}})_{1\hat{k}}$ contributes to the growth of $c_{m,A_i,1k}+c_{m,A_{i'},1(d_{i'}-k+1)}$, when $\hat{i}=i'$ and $d_{i'}-p+1 \leq \hat{k} \leq d_{i'} $. 

\end{itemize}

\noindent Putting all the contributions together $c_{m',A'_i,1\gamma'(k)}+c_{m',A'_{i'},1\gamma(k)}<c_{m,A_i,1k}+c_{m,A_{i'},1(d_{i'}-k+1)}$ as we wanted .

Adding all the inequalities and equalities from (a), (b), (c), (d) and (e) together we obtain that $c_{m'}<c_{m}$. 

\

\noindent (ii) We will proceed to compare the contributions that the entries of $A$ make to the number of inversions  of $A$ against the contributions that the entries of $A'$ make to the number of inversions of $A'$. 
 
\noindent (a) $e_{A',\hat{j}\hat{k}}=e_{A,\hat{j}\hat{k}}$ whenever $(\hat{k},\hat{j})<_{lex} (k_1,j_1)$ or $(k_2,j_2)<_{lex} (\hat{k},\hat{j})$. Since after a permutation the entries $A_{jk}$ with $(\hat{k},\hat{j})<_{lex} (k,j)$ become the entries $A'_{jk}$ with $(\hat{k},\hat{j})<_{lex} (k,j)$.

\noindent (b) $e_{A',\hat{j}\hat{k}} \leq e_{A,\hat{j}\hat{k}}$ for $(k_1,j_1) <_{lex} (\hat{k},\hat{j}) <_{lex} (k_2,j_2)$. Since $A_{jk} \geq A'_{jk}$ for $(k,j) >_{lex} (k_1,j_1)$.

\noindent (c) $e_{A',j_1k_1}+e_{A',j_2k_2} < e_{A,j_1k_1}+e_{A,j_2k_2}$. We will break down the explanation:
\begin{itemize}[-]
\item If $A'_{\hat{j}\hat{k}} > A'_{j_2k_2}$, with $(k_2,j_2) <_{lex} (\hat{k},\hat{j})$, then $A_{\hat{j}\hat{k}}=A'_{\hat{j}\hat{k}} > A'_{j_2k_2}= A_{j_1k_1}$, and $(k_1,j_1) <_{lex} (\hat{k},\hat{j})$. Also if $A'_{\hat{j}\hat{k}} > A'_{j_1k_1}$, with $(k_2,j_2) <_{lex} (\hat{k},\hat{j})$, then $A_{\hat{j}\hat{k}}=A'_{\hat{j}\hat{k}}> A'_{j_1k_1} =A_{j_2k_2}$.

\noindent This proves that if $A'_{\hat{j}\hat{k}}$ contributes to the growth of $e_{A',j_1k_1}+e_{A',j_2k_2}$  then $A_{\hat{j}\hat{k}}$ contributes to the growth of $e_{A,j_1k_1}+e_{A,j_2k_2}$ when $(\hat{k},\hat{j}) >_{lex} (k_2,j_2)$.

\item If $A'_{\hat{j}\hat{k}} > A'_{j_1k_1}$, with $(k_1,j_1)<_{lex} (\hat{k},\hat{j}) <_{lex} (k_2,j_2)$ , then $A_{\hat{j}\hat{k}} = A'_{\hat{j}\hat{k}} > A'_{j_1k_1}=A_{j_2k_2} > A_{j_1k_1}$. 

\noindent This proves that if $A'_{\hat{j}\hat{k}}$ contributes to the growth of $e_{A',j_1k_1}+e_{A',j_2k_2}$  then $A_{\hat{j}\hat{k}}$ contributes to the growth of $e_{A,j_1k_1}+e_{A,j_2k_2}$, when $(k_1,j_1)<_{lex} (\hat{k},\hat{j}) <_{lex} (k_2,j_2)$.

\item Finally, notice that $A'_{j_2k_2} \ngtr A'_{j_1k_1}$, while $A_{j_2k_2} > A_{j_1k_1}$. Hence $A'_{j_2k'_2}$ does not contribute to the growth of $e_{A',j_1k_1}+e_{A',j_2k_2}$ , but $A_{j_2k_2}$ contributes to the growth of $e_{A,j_1k_1}+e_{A,j_2k_2}$.
\end{itemize}
Putting all the contributions together we obtain $e_{A',j_1k_1}+e_{A',j_2k_2} < e_{A,j_1k_1}+e_{A,j_2k_2}$ as we wanted. 

Adding all the inequalities and equalities described in (a), (b) and (c) together we obtain $e_{A'} < e_{A}$. 

\

\noindent (iii)
 Let $(B_s^{(m)}, \dots, B_1^{(m)}, B_0^{(m)})$ be a representation of $m$ in terms of inversion minimal matrices, and $(B_s^{(m')}, \dots, B_1^{(m')}, B_0^{(m')})$ be a representation of $m'$ in terms of inversion minimal matrices, then $B_{i'}^{(m)}=B_{i'}^{(m')}$ for $i' \neq i$, and the entries of $B_{i}^{(m')}$ are the same as the entries of $B_{i}^{(m)}$ after a permutation. This proves that $c_m=c_m'$ and that $e_{m',i'}=e_{m,i'}$ for $i' \neq i$.

To prove that $e_{m',i} < e_{m.i}$, (and thus obtain that $e_{m'} < e_m$): Let $h$ and $h'$ be such that $(B_{i}^{(m)})_{hk}=Y_{ijk}$ and $(B_{i}^{(m)})_{h'k}=Y_{ij'k}$. We will assume without loss of generality that $h < h'$. 

Let $A'_{i}$ be the matrix such that $(A'_{i})_{\hat{j}k}=(B^{(m)}_{i})_{\hat{j}k}$ whenever $\hat{j} \neq h, h'$, and $(A'_{i})_{hk}=Y_{ilk}$ , $(A'_{i})_{h'k}=Y_{il'k} $, then it is clear that $e_{B^{(m')}_{i}} \leq e_{A'_i}$ by the minimality of $B_{i}^{(m')}$. So we only need to prove that $e_{A'_i} < e_{B_{i}^{(m)}}$. 

Since $u_{ij} \nprec u_{ij'}$ and $u_{ij'} \nprec u_{ij}$, then there are $(h_1,k_1)= \text{min}\{ (\tilde{h},\tilde{k}) : \tilde{h}=h \text{ or } \tilde{h}=h' \text { and } \exists (\hat{h}, \hat{k})$, with $ ( \tilde{k}, \tilde{h} ) <_{lex} (\hat{k}, \hat{h} ), \hat{h}=h \text{ or } \hat{h}=h', \text{ and }(B_{i}^{(m)})_{ \tilde{h} \tilde{k} } < (B_{i}^{(m)})_{\hat{h}\hat{k}}  \}$, and $(h'_1,k'_1)= \text{min} \{(\tilde{h},\tilde{k}) :  \tilde{h}=h \text{ or } \tilde{h}=h', (k_1, h_1 ) <_{lex} (\tilde{k}, \tilde{h}) \text{ and } (B_{i}^{(m)})_{h_1k_1} < (B_{i}^{(m)})_{\hat{h}\hat{k}} \}$. 

Replace $B_{i}^{(m)}$ by $B^{(1)}_i$, a matrix with the property that $(B^{(1)}_i)_{h_1k_1}=(B_i)_{h_2k_2}$, $(B^{(1)}_i)_{h_2k_2}=(B_i)_{h_1k_1}$ and $(B^{(1)}_i)_{\hat{j}k}=(B^{(m)}_i)_{\hat{j}k}$, whenever $(\hat{j},k) \neq (h_1,k_1), (h_2,k_2)$.

If $B^{(1)}_i \neq A'_i$, we can apply the exchange described above again, and after a finite number of exchanges we must obtain that $B^{(\eta)}_{i}=A'_i$, and then by Remark \ref{1.7.}(ii), we have that $e_{A'_i} = e_{B^{(\eta)}_i} < e_{B^{(\eta -1)}_i} < \dots < e_{B^{(1)}_i} < e_{B^{(m)}_i} $, as we wanted. 

\

\noindent (iv) ($ \Rightarrow $)  We will proceed by contradiction. Assume $m$ is completely reduced and that there  are $T_{ij}$ and $T_{i'j'} $ that divide $m$ such that $u_{ij} \nprec u_{i'j'}$ and $u_{i'j'} \nprec u_{ij}$. Then, by either Remark \ref{1.7.}(i) or Remark \ref{1.7.}(iii), there is $m'$ such that $(c_{m'},e_{m'}) <_{lex} (c_m,e_m)$, which implies that $(c_m,e_m) \neq (0,0)$ contradicting the fact that $m$ is completely reduced. 

\noindent ( $\Leftarrow $) If  $u_{ij} \prec u_{i'j'}$ or $u_{i'j'} \prec u_{ij}$ whenever $T_{ij}T_{i'j'}$ divides $m$, then we can represent $m$ as:
\[ m=\prod_{i=0}^{s}  \left(  \prod_{j=1}^{r_i} T_{i\sigma_{i}^{m}(j)} \right) \text{ with } (\star) \text{ } u_{i\sigma_{i}^{m}(j)} \prec u_{i'\sigma_{i'}^{m}(j')} \text{ whenever } (i,j) <_{lex} (i',j') \]

Then the arrangement $(A_r, \dots, A_1, A_0)$, with $(A_i)_{jk}=Y_{i\sigma^{m}_{i}(j)k}$ is a representation of $m$ by matrices and because of $(\star)$, we get $c_m=0$ and $e_{A_i}=0$, for all $i$. This means that $A_i$ is inversion minimal, implying $(c_m,e_m)=(0,0)$, or equivalently, that $m$ is completely reduced.
\end{proof}

We will use the fact that the set $\{ \Psi(m): m \in \text{Mon}(S'), m \text{ completely reduced} \}$ is linearly independent over $K$. This is a direct consequence of the following lemma.

 \begin{Lemma} \label{1.8.} Let $m, m' \in \text{Mon}(S')$, with $m, m'$ completely reduced. If $\Psi(m) =\Psi (m')$, then $m=m'$.
 \end{Lemma}
 \begin{proof}
 
Let $(B_s^{(m)}, \dots, B_1^{(m)}, B_0^{(m)})$ and $(B_s^{(m')}, \dots, B_1^{(m')}, B_0^{(m')})$ be representations of $m$ and $m'$ by $(s+1)$-tuples of inversion minimal matrices, respectively, and let $\Psi(m)=\Psi(m')=X_1^{\alpha_1}\cdot X_2^{\alpha_2} \cdot \dots \cdot X_n^{\alpha_n} \cdot t_1^{\beta_{1}} \cdot t_2^{\beta_2} \cdot \dots \cdot t_s^{\beta_s}$. If we prove that $B_i^{(m)}=B_i^{(m')}$ for all $0 \leq i \leq s$ then we can conclude that $m=m'$.

Since $m=\prod_{i=0}^{s}  \left(  \prod_{j=1}^{r_i} T_{i\sigma_{i}^{(m)}(j)} \right)$, then $\Psi(m)=\prod_{i=0}^{s}  \left(  \prod_{j=1}^{r_i} \Psi(T_{i\sigma_{i}^{(m)}(j)}) \right)$ 

\[ \Psi(m)= \left( \prod_{j=1}^{r_0} Y_{0\sigma_0^{(m)}(j)} \right) \cdot \prod_{i=1}^{s}  \left(  \prod_{j=1}^{r_i} u_{i\sigma_i^{(m)}(j)}t_i \right)=\left( \prod_{j=1}^{r_0} Y_{0\sigma_0^{(m)}(j)} \right) \cdot \prod_{i=1}^{s}  \left(  \prod_{j=1}^{r_i} u_{i\sigma_{i}^{(m)}(j)} \right) \cdot \left( \prod_{i=1}^{s} t_i^{r_i} \right).\]

This implies that the size of $B_i^{(m)}$ is $\beta_i \times d_i$ for $1 \leq i \leq s$, and the size of $B_0^{(m)}$ is $\beta_0 \times 1$, where $\beta_0= \sum_{i=0}^{n} \alpha_{i} - \sum_{i=0}^{r} \beta_i \cdot d_i$. An analogous argument shows that the size of $B_i^{(m')}$ is $\beta_i \times d_i$ for $1 \leq i \leq s$, and the size of $B_0^{(m')}$ is $\beta_0 \times 1$, where $\beta_0= \sum_{i=0}^{n} \alpha_{i} - \sum_{i=0}^{r} \beta_i \cdot d_i$.

We will proceed by contradiction and assume $m \neq m'$. Then $(B_i^{(m)})_{jk} \neq (B_i^{(m')})_{jk}$, for some $(i,j,k)$. Let $i'= \text{max}\{i : \exists (j,k) \text{ such that } (B_i^{(m)})_{jk} \neq (B_i^{(m')})_{jk }\}$ and let $(j',k')$ be such that $(B^{(m)}_{i'})_{j'k'} \neq (B^{(m')}_{i'})_{j'k'}$ and $(B^{(m)}_{i'})_{jk} = (B^{(m')}_{i'})_{jk}$ for all $(k,j) <_{lex} (k',j')$. Without loss of generality we can assume that $(B_{i'}^{(m)})_{j'k'} < (B^{(m')}_{i'})_{j'k'}=X_{\tilde{i}}$. Since $m$ is completely reduced we know that $(B^{(m)}_i)_{jk} \leq (B^{(m)}_{i'})_{j'k'} < X_{\tilde{i}}$ for all $i<i'$ and all $(k,j) >_{lex} (k',j')$ when $i=i'$. This implies that the highest power of $X_{\tilde{i}}$ that divides $\Psi(m')$ is higher than the highest power of $X_{\tilde{i}}$ that divides $\Psi(m)$, hence $\Psi(m') \neq \Psi(m)$, which is a contradiction.
\end{proof}

We are now ready to prove Lemma \ref{1.4.}.

\begin{proof}[\textit{Proof of Lemma 1.4}]

To guarantee the existence of a monomial order, $<_{\tau}$, such that the underlined terms in $G$ are initial according to $<_{\tau}$, it is enough to prove that the reduction relation modulo $G$ is Noetherian, by Theorem \ref{1.6.}. 

Let $f \in S'$, we will define $ c_f= \sum_{m \in \text{supp}(f)} c_m$ and $ e_f=\sum_{m \in \text{supp}(f)} e_m$ and the level of reduction of $f$ as $(c_f,e_f)$. 

Let $f'$ be a one step reduction of $f$ modulo $G$, then $\text{supp}(f')=(\text{supp}(f)-\{m_0\})\cup \{m'_0\}$, for some $m_0 \in \text{supp}(f)$ such that $m'_0=\frac{T_{il}T_{i'l'}}{T_{ij}T_{i'j'}} m_0$, with $i \leq i' $, $u_{ij} \nprec u_{i'j'}$ and $u_{i'j'} \nprec u_{ij}$, $u_{ij}u_{i'j'}=u_{il}u_{i'l'}$ and $u_{il} \prec u_{i'l'}$. Then, by either Remark \ref{1.7.}(i) or Remark \ref{1.7.}(iii), we get $(c_{m'_0}, e_{m'_0}) <_{lex} (c_{m_0},e_{m_0})$.

Notice that:
\[ (c_{f'},e_{f'}) \leq_{lex}  (c_{m'_0},e_{m'_0})+\sum_{m \in (\text{supp}(f)-\{ m_0\})}(c_m, e_m) <_{lex} \sum_{m \in \text{supp}(f)}(c_m, e_m) = (c_f,e_f)\]

If the reduction relation modulo $G$ is not Noetherian, then there is an infinite set of polynomials $\{f, f^{(1)}, f^{(2)}, \dots \}$ such that $(c_f,e_f) >_{lex} (c_{f^{(1)}}, e_{f^{(1)}}) >_{lex} (c_{f^{(2)}},e_{f^{(2)}}) >_{lex} \dots $, but this contradicts one of the corollaries of Dickson's Lemma, which establishes that there are not infinitely decreasing sequences of terms in any monomial order, (see \cite[2.1.7]{HH}).

Hence $\rightarrow_{G}$ is indeed Noetherian and the existence of $<_{\tau}$ is guaranteed. Furthermore, by Remark \ref{1.7.}(iv), if $\hat{f}$ is a remainder of $f$ when dividing by $G$, then $m \in \text{supp}(\hat{f})$ if and only if $m$ is completely reduced. 

To prove that $G$ is  Gr\"obner basis with respect to $<_{\tau}$, index the elements of $G$ as $g_1,g_2, \dots, g_{\kappa}$, we will denote the $S'$-polynomial of $g_i,g_j$ by $g_{i,j}$. 

Let $\hat{g_{i,j}}$ be a remainder of $g_{i,j}$ when divided by $G$. It is known that $g_{i,j}=\left(\sum_{k=1}^{\kappa} f_kg_k \right)+ \hat{g_{i,j}}$. Applying $\Psi$ on both sides of the equation, along with the fact that $G \subseteq \text{Ker}\Psi$, we obtain $0=\Psi(\hat{g_{i,j}})$. And this implies $\hat{g_{i,j}}=0$; for if $\hat{g_{i,j}} \neq 0$, then $\hat{g_{i,j}}$ is a linear combination of completely reduced monomials, and by Lemma \ref{1.8.} their images are linearly independent over $K$ which contradicts the fact that $0=\Psi(\hat{g_{i,j}})$.

We have proved that all $S'$-polynomials of $G$ reduce to $0$ when divided by $G$, so Buchberger's criterion (see \cite{CLO}) asserts that $G$ is a Gr\"obner basis according to $<_{\tau}$.
\end{proof}

\section{Main Result and Examples}

We will establish again the setup for our main result. Let $1 = d_0 \leq d_1 \leq \dots \leq d_s$, $I_0$ the maximal ideal of $S$ and $I_i$ homogeneous monomial ideals with minimal generators of fixed degree $d_i$ for $1 \leq i \leq s$. Let the minimal monomial generators of $I_i$  be $u_{ij}$, for $0 \leq j \leq n_{i}$, with the additional condition $u_{i1} >_{rev} u_{i2} >_{rev} \dots >_{rev} u_{in_{i}}$. Let $M=I_1 \oplus I_2 \oplus \dots \oplus I_s$, $\mathcal{M}=\coprod_{i=0}^{s} \{u_{ij}: 1\leq j \leq n_i \}$ and $\mathcal{N}=\{u_{ij}t_i: 0 \leq i \leq s, 1 \leq j \leq n_i \}$ with the convention that $t_0=1$. Consider the surjective map $\Psi: K[T_{ij}: u_{ij} \in \mathcal{M}] \longrightarrow K[\mathcal{N}] \cong \mathcal{R}(M)$, given by $\Psi(T_{ij})=u_{ij}t_{i}$ for all $u_{ij} \in \mathcal{M}$.

 \begin{Lemma}\label{2.1.}
If $\mathcal{M}$ is closed under comparability $Ker \Psi$ has a squarefree quadratic Gr\"obner basis.\end{Lemma}
\begin{proof}

Consider $G$ as in Lemma \ref{1.4.}. We already proved that $( G ) \subseteq \textrm{Ker}\Psi$. 

To prove the other inclusion: Let $f \in S'$, and assume $f$ does not reduce to 0 with respect to $G$, (i.e $f \notin (G)$).  Let $f'$ be the remainder of $f$ when divided by $G$.

Hence $\Psi(f)=\Psi(f') \neq 0$, since $f'$ is a linear combination of completely reduced monomials and, by Lemma \ref{1.8.}, their images are linearly independent over $K$,.

We conclude that if $f$ does not reduce to $0$ with respect to $G$, (equivalently $f \notin (G)$), then $f \notin \textrm{Ker}\Psi$. Proving that $\textrm{Ker}\Psi=(G)$ has a Gr\"{o}bner basis of quadrics given by $G$. 

To prove that $\textrm{Ker}\Psi=(G)$ is squarefree, notice that the initial term of $\underline{T_{ij}T_{i'j'}}-T_{il}T_{i'l'}$ is a square only if  $(i,j)=(i',j')$. Since $(u_{il}, u_{i'l'})=\text{sort}(u_{ij}, u_{ij})=(u_{ij}, u_{ij})$, we conclude $\underline{T_{ij}T_{i'j'}}-T_{il}T_{i'l'}=0$. \end{proof}

\begin{Theorem} \label{2.2.} If $\mathcal{M}$ is closed under comparability then $\displaystyle  K[T_{ij}: u_{ij} \in \mathcal{M}]/\textrm{Ker}\Psi \cong K[\mathcal{N}]   \cong \mathcal{R}(M)$ is Koszul and a normal Cohen-Macaulay domain.\end{Theorem}

\begin{proof}

The fact that $\textrm{Ker}\Psi=(G)$, implies that the defining ideal of $\mathcal{R}(M)$ has a quadratic Gr\"obner basis that is square free. Because $\text{Ker}\Psi$ satisfies the G-quadratic condition, we conclude that $ \mathcal{R}(M)$ is Koszul. 

Additionally, the fact that the Gr\"obner basis is squarefree implies that $K[\mathcal{N}]$ is a normal domain, due to a result by Sturmfels \cite{S}. Another result states that $K[\mathcal{N}]$ is a normal domain if and only if $\mathcal{N}$ is an affine normal semigroup. Finally, a result by Hochster \cite{Hoc} proves that  if $\mathcal{N}$ is an affine normal semigroup then $ \mathcal{R}(M) \cong K[\mathcal{N}]$ is Cohen-Macaulay. \end{proof}

We will proceed to present a thorough characterization of sets that are closed under comparability, and correspond to Rees and multi-Rees algebras, to obtain corresponding Koszul and Cohen-Macaulay normal domains via Theorem \ref{2.2.}. We will denote by $\mathcal{M}_i$ the set of minimal monomial generators for $I_i$. Our first step is to prove that if $\mathcal{M}$ is closed under comparability, then $\mathcal{M}_i$ has to be a principal strongly stable set. We will start by recalling a set of equivalences given in a Lemma due to De Negri. 

\begin{Lemma} [\textit{A description of principal strongly stable sets}]\label{2.3.} 
Let $w_1= X_1^{\alpha_1}X_2^{\alpha_2} \dots X_n^{\alpha_n}$ and $w_2=X_1^{\beta_1}X_2^{\beta_2} \dots X_n^{\beta_n}$ be two monomials of the same degree. Then $w_2 \in \mathcal{B}(w_1)$ if and only if $\sum_{i=k}^n \beta_i \leq \sum_{i=k}^n \alpha_i$ for $2 \leq k \leq n$.
\end{Lemma}

The proof of this lemma can be found in \cite[Lemma 1.3]{D}.

Next, we will show that the minimal generating set of a principal strongly stable ideal is closed under sorting. For this we will first describe the sorting of two monomials in a different way. 

Let $u_1=X_1^{\alpha_1}X_2^{\alpha_2}\dots X_n^{\alpha_n}$ and $u_2=X_1^{\beta_1}X_2^{\beta_2}\dots X_n^{\beta_n}$ be two monomials of the same degree. Let $J=\{i: \alpha_i+\beta_i \text{ is odd }\}$, then the cardinality of $J$ is an even number, $2t$, and its elements can be indexed as $i_1 < i_2 < \dots < i_{2t}$. 

It is easy to see that $\text{sort}(u_1,u_2)=(X_1^{\alpha'_1}X_2^{\alpha'_2}\dots X_n^{\alpha'_n}, X_1^{\beta'_1}X_2^{\beta'_2}\dots X_n^{\beta'_n} )$, where $\alpha'_{i}= \beta'_{i}=\frac{\alpha_{i}+\beta_{i}}{2}$ if $i \notin J$ and  $\alpha'_{i_r}=\frac{\alpha_{i_r}+\beta_{i_r}+(-1)^{r+1}}{2}$ and $\beta'_{i_r}=\frac{\alpha_{i_r}+\beta_{i_r}+(-1)^r}{2}$. 

Furthermore, $\sum_{k=i}^n \alpha'_k = \sum_{k=i}^n \beta'_k=\sum_{k=i}^n \frac{\alpha_k+\beta_k}{2}$ if $i>i_{2t}, i\leq i_1, \text{ or }, i_{2r} < i \leq i_{2r+1}$. On the other hand, if $i_{2r-1} < i \leq i_{2r}$ then $\sum_{k=i}^n \alpha'_k =-\frac{1}{2}+ \sum_{k=i}^n \frac{\alpha_k+\beta_k}{2}$ and $ \sum_{k=i}^n \beta'_k=\frac{1}{2}+\sum_{k=i}^n \frac{\alpha_k+\beta_k}{2}$.

\begin{Lemma}\label{2.4.}
Let $u_1,u_2 \in \mathcal{B}(w)$, then $\overline{u_1},\overline{u_2} \in \mathcal{B}(w)$, with $(\overline{u_1},\overline{u_2})=\text{sort}(u_1,u_2)$.
\end{Lemma}

\begin{proof}
Let $u_1=X_1^{\alpha_1}X_2^{\alpha_2}\dots X_n^{\alpha_n}$, $u_2=X_1^{\beta_1}X_2^{\beta_2}\dots X_n^{\beta_n}$ and $w=X_1^{\gamma_1}X_2^{\gamma_2}\dots X_n^{\gamma_n}$. Since $u_1,u_2 \in \mathcal{B}(w)$,  we get that $ \sum_{k=i}^n \alpha_k, \sum_{k=i}^n \beta_k \leq \sum_{k=i}^n \gamma_k$ for $i=2, \dots, n$, because of Lemma \ref{2.3.}. Thus $\sum_{k=i}^n \frac{ \alpha_k+\beta_k}{2} \leq \sum_{k=i}^n \gamma_k$.

This implies immediately that  $\sum_{k=i}^n \alpha'_k \leq \sum_{k=i}^n \gamma_k $ for all $i=2, \dots, n$, by the previous observation.
 
Realize now that $\sum_{k=i}^n  (\alpha_k+\beta_k)$ is an odd number when $i_{2s-1} < i \leq i_{2s}$, so the fact that $\sum_{k=i}^n (\alpha_k+\beta_k) \leq 2\sum_{k=i}^n \gamma_k$  implies that $1+\sum_{k=i}^n (\alpha_k+\beta_k) \leq 2\sum_{k=i}^n \gamma_k$, and this in turn shows that $\sum_{k=i}^n \beta'_k \leq  \sum_{k=i}^n \gamma_k $  for all $i=2, \dots, n$.

And we get that $\overline{u_1},\overline{u_2} \in \mathcal{B}(w)$ because of Lemma \ref{2.3.}.
\end{proof}

The following theorem gives a complete characterization of sets closed under comparability. 

\begin{Theorem}\label{2.5.}
Consider a sequence of natural numbers $1 \leq d_1 \leq d_2 \leq \dots \leq d_s$, and monomial homogeneous ideals $I_i$ generated in fixed degree $d_i$.

Let $\mathcal{M}_0=\{X_1,X_2, \dots, X_n\}$ and let $\mathcal{M}_i$ be the minimal monomial generating set of $I_i$. Index the elements of $\mathcal{M}_i$ as $u_{i1}>_{rev} u_{i2} >_{rev} \dots >_{rev} u_{in_i}$.

Then $\mathcal{M}=\coprod_{i=0}^{r} \mathcal{M}_i$ is closed under comparability, if and only if, (i) $\mathcal{M}_i=\mathcal{B}(u_{in_i})$ and (ii) $\text{max}(u_{in_i}) \leq \text{min}(u_{(i+1)n_{i+1}})$ for all $1 \leq i \leq s-1$.  
\end{Theorem}

\begin{proof}

\

($\Rightarrow$) We will first prove that $\mathcal{M}_i \subseteq \mathcal{B}(u_{in_i})$ by contradiction. 

Let $N=\{ j : u_{ij} \notin \mathcal{B}(u_{in_i}) \}$, if $ \mathcal{M}_i \nsubseteq \mathcal{B}(u_{in_i})$ then $N \neq \emptyset $. Let $j'=\text{max}(N)$. 

Consider $u_{in_i}= X_1^{\alpha_1}X_2^{\alpha_2}\dots X_n^{\alpha_n}$ and $u_{ij'}=X_1^{\beta_1}X_2^{\beta_2}\dots X_n^{\beta_n}$, and let $p=\text{max}\{j: \alpha_j \neq \beta_j \}$, then $\alpha_p > \beta_p$ since $ u_{in_i} <_{rev} u_{ij'}$. Additionally, since $u_{ij'} \notin \mathcal{B}(u_{in_i})$, the set $N'= \{k: \sum_{i=k}^n \alpha_i < \sum_{i=k}^n \beta_i\} \neq \emptyset$, because of Lemma \ref{2.3.}. Let $q=\text{max}(N')$.

Let $(\overline{u}, \overline{v})=\text{sort}(u_{in_i},u_{ij'})$, then $\overline{v} \in \mathcal{M}_i$, since $\mathcal{M}$ is closed under comparability. So  $\overline{v}=u_{ij''}$ for some $j'' \leq n_i$. Let $\overline{v}=u_{ij''}=X_1^{\beta'_1}\dots X_n^{\beta'_n}$. Notice that either $\sum_{k=q}^n \beta'_k=\sum_{k=q}^n \frac{\alpha_k+\beta_k}{2}$ or $\sum_{k=q}^n \beta'_k=\frac{1}{2}+\sum_{k=q}^n \frac{\alpha_k+\beta_k}{2}$. Regardless of the case, $\sum_{k=q}^n \beta'_k > \sum_{k=q}^n \alpha_k$, which implies that $u_{ij''} \notin \mathcal{B}(u_{in_i})$ by Lemma \ref{2.3.}. 

On the other hand, $\beta'_i=\beta_i$ when $i>p$, and either $\beta'_i=\frac{\alpha_p+\beta_p+1}{2}$ or $\beta'_i=\frac{\alpha_p+\beta_p}{2}$, because of the maximality of $p$. Regardless, $\overline{v}=u_{ij''} <_{rev} u_{ij'}$, which implies $j''>j'$. Contradicting the maximality of $j'$.

\

Now we will prove that $\mathcal{B}(u_{in_i}) \subseteq  \mathcal{M}_i$. 

Let $u=\prod_{j=1}^{p}X_j^{\gamma_j} \in \mathcal{M}_i$ with $\gamma_{p} \neq 0$, we will define $u^{*}=\prod_{j=1}^{p} X_j^{\gamma^{*}_j} $ with $\gamma^{*}_{j} =\gamma_{j}$ for $j<p-1$, $\gamma^{*}_{p-1}=\gamma_{p-1}+1$ and $\gamma^{*}_{p}=\gamma_{p}-1$. Notice that $u <_{rev} u^{*}$. The fact that $\mathcal{M}$ is closed under comparability implies that $u^{*} \in \mathcal{M}_i$, since $\text{ord}(X_{p-1}, u)=(X_{p},u^{*})$. The repeated application of this fact proves that $u_{i1}=X_1^{d_i}$, and it also proves that if $u=X_1^{d_i-k}X_2^k \in \mathcal{M}_i$, then $\mathcal{B}(u) \subset \mathcal{M}_i$.

Let $k=\text{max}\{j \leq n_i : \mathcal{B}(u_{ij}) \subset \mathcal{M}_i \}$, notice that $k \geq 1$, since $u_{i1}=X_1^{d_i}$. We will prove that $k=n_i$ by contradiction.

If $k<n_i$, let $u_{i(k+1)}=\prod_{j=1}^{p}X_j^{\beta_j}$ with $\beta_{p} \neq 0$. Notice $p \geq 3$, since $ \mathcal{B}(u_{i(k+1)}) \nsubseteq \mathcal{M}_i$. Additionally, the set $N=\{v  \notin \mathcal{M}_i : v \in  \mathcal{B}(u_{i(k+1)}) \} \neq \emptyset$, so we can denote $w=\prod_{j=1}^{p} X_j^{\gamma_j}=\text{max}_{<_{rev}}N$. Remember that  $u_{i(k+1)} <_{rev} (u_{i(k+1})^{*} \in \mathcal{M}_i$, which implies that $(u_{i(k+1})^{*}=u_{il}$ for some $l < k+1$. 

Notice that $\gamma_p=\beta_p$. Otherwise $\gamma_p \leq \beta_p-1= \beta^{*}_{p}$, and then $\sum_{j=s}^n \gamma_j =0 = \sum_{j=s}^n \beta^{*}_j$ when $s>p$, $\sum_{j=p}^n \gamma_j = \gamma_p \leq \beta^{*}_{p} = \sum_{j=p}^n \beta'_j$ and $\sum_{j=s}^n \gamma_j \leq \sum_{j=s}^n \beta_j =\sum_{j=s}^n \beta^{*}_j $ when $s<p$, which by Lemma \ref{2.3.} implies that $w \in \mathcal{B}(u_{il}) \subset \mathcal{M}_i$, which is false.

A similar argument shows that $w^{*}=(\prod_{j=1}^{p-2}X_j^{\beta_{j}})X_{p-1}^{\beta_{p-1}+1}X_p^{\beta_p-1} \in \mathcal{B}(u_{il}) \subset \mathcal{M}_i$. Let $w'= (\prod_{j=1}^{p-3}X_j^{\beta_{j}})X_{p-2}^{\beta_{p-2}+1}X_{p-1}^{\beta_{p-1}-1}X_p^{\beta_p}$, then $w' \in \mathcal{B}(u_{i(k+1)})$ and $w<_{rev} w'$, which implies that $w' \in \mathcal{M}_i$, by the maximality of $w$. 

So $\text{sort}(w',w^{*})= ((\prod_{j=1}^{p-3}X_j^{\beta_{j}})X_{p-2}^{\beta_{p-2}+1}X_{p-1}^{\beta_{p-1}}X_p^{\beta_p-1} , w)$, and this implies that $w \in \mathcal{M}_i$, since $\mathcal{M}$ is closed under comparability. But this is a contradiction, and we obtain that $k=n_i$. 

\

Finally, if $\text{max}(u_{in_i}) > \text{min}(u_{(i+1)n_{i+1}}) $, let $b_i$ and $s_i$ be such that $\text{max}(u_{in_i})=X_{b_i}$ and $\text{min}(u_{in_i})=X_{s_i}$. Then $b_i < s_{i+1}$ and $u_{in_i}= \prod_{j=b_i}^{n} X_j^{\alpha_j}$, with $\alpha_{b_i} \neq 0$, while $u_{(i+1)n_{i+1}}=\prod_{j=1}^{s_{i+1}} X_j^{\beta_j}$, with $\beta_{s_{i+1}} \neq 0$. Let $\text{ord}(u_{in_i},u_{(i+1)n_{i+1}})=(\prod_{j=t}^{n} X_{j}^{\delta_j}, \prod_{j=1}^{t} X_{j}^{\delta'_j})$, then the fact that $\sum_{j=t}^{n} \delta_j = d_i = \sum_{j=b_i}^{n} \alpha_j$ implies that $t \leq s_{i+1}$, furthermore since $\alpha_{b_i} \neq 0 \neq \beta_{s_{i+1}}$ and $b_i < s_{i+1}$, then $\delta_{s_{i+1}} > \alpha_{s_{i+1}}$, which means that if $\prod_{j=t}^{n} X_{j}^{\delta_j} \in \mathcal{M}_i$ then $u_{in_i} >_{rev} \prod_{j=t}^{n} X_{j}^{\delta_j}$, which is a contradiction to the minimality of $u_{in_i}$.

Hence, if $\mathcal{M}$ is closed under comparability then $\text{max}(u_{in_i}) \leq \text{min}(u_{(i+1)n_{i+1}})$.

\

($\Leftarrow$) If $\mathcal{M}_i=\mathcal{B}(u_{in_i})$, then $\mathcal{M}$ is closed under sorting by Lemma \ref{2.4.}. 

So we only have to prove that $\mathcal{M}$ is closed under ordering. Let $l>k$, so $d_l \geq d_k$,  and consider $v \in \mathcal{M}_{l}=\mathcal{B}(u_{ln_l})$ with $u_{ln_l}=\prod_{j=b_l}^{s_l} X_j^{\alpha_j}$, and $\alpha_{b_l} \neq 0 \neq \alpha_{s_l}$ and $u \in \mathcal{M}_{k}=\mathcal{B}(u_{kn_k})$ with  $u_{kn_k}=\prod_{j=b_k}^{s_k} X_j^{\beta_j}$ and $\beta_{b_k} \neq 0 \neq \beta_{s_k}$. Since $\text{min}(u_{ln_l}) \geq \text{max}(u_{kn_k})$,  it must happen that $s_l \leq b_k$.

We can express $v$ and $u$ in the following way $v=\prod_{j=1}^{s_l} X_j^{\alpha'_j}$ and $u=\prod_{j=1}^{s_k} X_j^{\beta'_j}$. We want to prove that $\hat{u} \in \mathcal{M}_k=\mathcal{B}(u_{kn_k})$ and that $\hat{v} \in   \mathcal{M}_{l}=\mathcal{B}(u_{ln_l})$, with $(\hat{u}, \hat{v})=\text{ord}(u,v)=(\prod_{j=t}^{s_k} X_j^{\delta_j}, \prod_{j=1}^{t} X_j^{\delta'_j})$. Notice that $t \leq s_l$, since $\sum_{j=1}^{s_l} \alpha'_{j}=d_l =\sum_{j=1}^{t} \delta'_j$ and that $\delta'_{j}=\alpha'_{j}+\beta'_j$ if $j<t$, $\delta'_t+\delta_t=\alpha'_t+\beta'_t$, $\delta_j=\alpha'_j + \beta'_j$ if $ t< j \leq s_l$ and $\delta_j=\beta'_j$ if $s_l < j \leq s_k$.

Applying Lemma \ref{2.3.} after the following observations proves that $\hat{u} \in \mathcal{M}_k=\mathcal{B}(u_{kn_k})$:

- $\sum_{j=r}^{s_k} \delta_j \leq \sum_{j=t}^{s_k}  \delta_j=d_k=\sum_{j=b_k}^{s_k} \beta_j$ if $t \leq r \leq s_l$.  Remember $s_l \leq b_k$.

- $\sum_{j=r}^{s_k} \delta_j = \sum_{j=s}^{s_k} \beta'_j \leq \sum_{j=\text{max}(s,b_k)} \beta_j$ if $s_l < r \leq s_k$.

To prove that $\hat{v} \in \mathcal{M}_l = \mathcal{B}(u_{ln_l})$, we need the following fact $\sum_{j=1}^{s_k} \beta'_j=d_k=\sum_{j=t}^{s_k} \delta_j=\sum_{j=t}^{s_l} \alpha'_j +\sum_{j=t}^{s_k}\beta'_{j}-\delta'_t$, or equivalently $\delta'_t=\sum_{j=t}^{s_l} \alpha'_j - \sum_{j=1}^{t-1} \beta'_j$ and then apply Lemma \ref{2.3.} to the following observations:

- $\sum_{j=r}^{t} \delta'_j \leq \sum_{j=1}^{t}  \delta'_j=d_l=\sum_{j=b_l}^{s_l} \alpha_j$ if $r \leq b_l$.  Remember $t \leq s_l$.

- $\sum_{j=r}^{t} \delta'_j = \sum_{j=r}^{t-1} (\alpha'_j+\beta'_j) +\delta'_t = \sum_{j=r}^{s_l} \alpha'_j - \sum_{j=1}^{r-1} \beta'_j \leq \sum_{j=r}^{s_l} \alpha'_j \leq \sum_{j=s}^{s_l} \alpha_j$ if $b_l <r $.
\end{proof}

We can now present examples:

\begin{Example} [\textit{Rees Algebras of principal strongly stable Ideals}] \label{2.6.} 
Let $s=1$, consider $I=(\mathcal{B}(u))$, the principal strongly stable ideal generated by the monomial $u$. Then Theorems \ref{2.2.} and \ref{2.5.} guarantee that the Rees Algebra $\mathcal{R}(I)$ is a Koszul, Cohen-Macaulay normal domain.

\end{Example}

\begin{Example} [\textit{Multi-Rees Algebras of particular principal strongly stable ideals}]\label{2.7.}

Consider $I_i=(\mathcal{B}(u_{i}))$, with the constraints $\text{max}(u_i) \leq \text{min} (u_{i+1})$ and $\text{deg}(u_{i+1}) \geq \text{deg}(u_i) $. 

Then Theorems \ref{2.2.} and \ref{2.5.} guarantee that the multi-Rees algebra of the module $M =I_{1} \bigoplus I_{2} \bigoplus \dots \bigoplus I_{s}$ is a Koszul, Cohen-Macaulay normal domain. In particular, by letting $u_{i}=X_n^{d_i}$, we obtain that the multi-Rees algebra generated by the direct sum of powers of the maximal ideal of $S$ is a Koszul Cohen-Macaulay normal domain, a fact proved by Lin and Polini in \cite{LPo}.
\end{Example}

We conclude with the following observation that produces examples of Koszul Cohen-Macaulay normal domains for certain semigroup rings, and with the appropriate notation, for the special fiber ring of $\mathcal{R}(M)$, $\mathcal{F}(M)=K[\mathcal{N}']=\mathcal{R}(M)/(X_1,\dots,X_n)\mathcal{R}(M)$, with $\mathcal{N}'$ defined as in page 3. 

Let $1 \leq d_1 \leq d_2 \leq \dots \leq d_s <m$ and let $S^{*}=K[X_1, \dots, X_n, X_{n+1}, \dots, X_{n+s}]$. For each $1 \leq i \leq s$ let $\mathcal{M}_i^{*}\subseteq \text{Mon}(S^{*}_{d_i})$, with $\vert \mathcal{M}_i^{*} \vert=n_i$, and such that if $u \in \mathcal{M}_i^{*}$ then $X_j \nmid u$ for all $j>n$. Index the elements of $\mathcal{M}_i^{*}$ as $u_{i1} >_{rev} u_{i2} >_{rev} \dots >_{rev} u_{in_i}$ and let $\mathcal{N}^{*}=\{u_{ij}X_{n+i}^{m-d_i}: 1 \leq i \leq s, 1 \leq j \leq n_i \}$. If $\mathcal{M}^{*}=\coprod_{i=1}^{s} \mathcal{M}_i^{*}$ is closed under comparability and we consider the surjective map $\Psi^{*}:K[T_{ij}: u_{ij} \in \mathcal{M}^{*}] \rightarrow K[\mathcal{N}^{*}]$ given by $\Psi^{*}(T_{ij})=u_{ij}X_{n+i}^{m-d_i}$ then Lemma \ref{1.4.}, Remark \ref{1.7.}, Lemma \ref{1.8.} and Lemma \ref{2.1.} still hold and their proofs are similar to ones presented under the appropriate modifications. An equivalent to Theorem \ref{2.2.} can be stated for this case as 

$K[\mathcal{N}^{*}] \cong K[T_{ij}: u_{ij} \in \mathcal{M}^{*}] / \text{Ker}\Psi^{*}$ \textit{is a Koszul, normal Cohen-Macaulay domain}.
 
 Additionally we get the following modification for Theorem \ref{2.5.}:

(1) \textit{$\mathcal{M}^{*}$ is closed under comparability if $\mathcal{M}^{*}_i =\mathcal{B}(u_{in_i})$ and $\text{max}(u_{in_i}) \leq \text{min}(u_{(i+1)n_{i+1}})$}

(2) \textit{If $\mathcal{M}^{*}$ is closed under comparability then $\mathcal{M}^{*}_i \subseteq \mathcal{B}(u_{in_i})$ and $\text{max}(u_{in_i}) \leq \text{min}(u_{(i+1)n_{i+1}})$}

And its proof is analogous to that of the applicable parts of Theorem \ref{2.5.}.

The fact that the condition $\mathcal{M}_i^{*}=\mathcal{B}(u_{in_i})$ and $\text{max}(u_{in_i}) \leq \text{min}(u_{(i+1)n_{i+1}})$ is sufficient but not necessary is evident by considering the example with $d_1=2, d_2=3, m=4, \mathcal{M}_1^{*}=\{ X_3^2, X_3X_4,X_3X_5, X_4X_5 \}, \mathcal{M}_2^{*}=\{ X_1^3, X_1^2X_3 \}$, which makes $\mathcal{M}^{*}$ to be closed under comparability, while $\mathcal{M}_1^{*} \neq \mathcal{B}(X_4X_5)$ and $\mathcal{M}_2^{*} \neq \mathcal{B}(X_1X_3^2)$, yet they are both generating sets of Veronese type, and $K[X_1^3X_7, X_1^2X_3X_7,X_3^2X_6^2, X_3X_4X_6^2, X_3X_5X_6^2, X_4X_5X_6^2]$ a Koszul, Cohen-Macaulay, normal domain. 

\bibliographystyle{plain}
\bibliography{refPRF}

\end{document}